\documentclass[a4paper,10pt]{article}
\usepackage{graphicx}
\usepackage{float}
\usepackage{amsthm}
\usepackage{amsfonts}
\usepackage[latin1]{inputenc}
\usepackage{indentfirst}
\usepackage{amsmath}
\usepackage{amssymb}
\usepackage{color}

\newtheorem{theo}{Theorem}[section]

\newtheorem{lemma}[theo]{Lemma}
\newtheorem{cor}[theo]{Corollary}
\newtheorem{prop}[theo]{Proposition}

\newtheorem{ann}{Assumption}

\numberwithin{equation}{section}

\def\Z{\mathbb{Z}}
\def\N{\mathbb{N}}

\def\P{\mathbb{P}}

\def\E{\mathbb{E}}

\def\R{\mathbb{R}}

\def\t{\textrm}
\def\d{\textrm{d}}

\def\w{\widetilde}

\def\ind{{\mathchoice {\rm 1\mskip-4mu l} {\rm 1\mskip-4mu l}
{\rm 1\mskip-4.5mu l} {\rm 1\mskip-5mu l}}}

\newcommand{\be} {\begin{equation}}
\newcommand{\ee} {\end{equation}}
\newcommand{\bea} {\begin{eqnarray}}
\newcommand{\eea} {\end{eqnarray}}
\newcommand{\Bea} {\begin{eqnarray*}}
\newcommand{\Eea} {\end{eqnarray*}} 

\begin{document}
\title{Lower large deviations for supercritical branching processes in random environment}
\author{Vincent Bansaye and Christian B\"oinghoff}
\date{revised version, 2016}
\maketitle \vspace{1.5cm}

\begin{abstract} 
Branching Processes in Random Environment (BPREs) $(Z_n:n\geq0)$ are the generalization of Galton-Watson processes where 
in each generation the reproduction law is picked randomly in an i.i.d. manner. In  the supercritical regime,  
 the process survives with a positive probability and grows exponentially on the non-extinction event. We  
 focus on rare events when the process takes positive values but lower than expected.  \\
More precisely, we are interested in  the lower large deviations of $Z$, which means the asymptotic behavior of the probability $\{1 \leq Z_n \leq \exp(n\theta)\}$ as $n\rightarrow \infty$. We
provide an expression of the  rate of decrease of this probability,  under some moment assumptions, which yields the rate function. This result
generalizes the lower large deviation theorem of  Bansaye and Berestycki (2009) by considering processes where $\P(Z_1=0 \vert Z_0=1)>0$ and also much weaker moment assumptions. 
\end{abstract}


\maketitle

\section{Introduction}
Branching processes in random environment (BPREs), which have been introduced in  \cite{smith69,athreya71}, are a discrete time and 
discrete size model in population dynamics. The model describes the development of a population of individuals which are exposed to a (random) environment. 
The environment influences the reproductive success of each individual in a generation. More formally, we can describe a BPRE as a two-stage experiment:\\
In each generation, an offspring distribution is picked at random and, 
given all offspring distributions (the environment), all individuals reproduce independently. \\

Special properties of the model like the problems of rare events and large deviations have been studied recently \cite{kozlov06, bansaye08,BK09,kozlov10, BB10,  HuangLiu}. 
In the Galton Watson case,  large deviations problems are  studied from a long time   \cite{athreya, athreya2} and fine  results have been obtained, see 
 \cite{FVlowerLDGW, FlWa, Ney, Rouault}.

Let us now define the branching process $Z$ in random environment. For this, let $\Delta$ be the space of all probability measures
 on $\N_0=\{0,1,2,\ldots\}$ (the set of possible offspring distributions) and let $Q$ be a random variable taking values in $\Delta$. 
By  $$m_q=\sum_{k\geq 0} k\ q(\{k\})\ ,$$
we denote the mean number of offsprings of $q\in \Delta$. Throughout the paper, we will shorten
$q(\{\cdot\})$ to $q(\cdot)$.  An infinite sequence
$\mathcal{E}=(Q_1,Q_2,\ldots)$ of independent, identically distributed (i.i.d.) copies of $Q$ is called a random environment. \\
The process $(Z_n : n\geq 0)$ with values in $\mathbb{N}_0$ is called a branching process in 
the random environment $\mathcal{E}$ if $Z_0$ is independent of $\mathcal{E}$ and it satisfies 
\begin{equation}  \label{transition} 
    \mathcal{L} \big(Z_{n} \; \big| \; \, \mathcal{E}, \  Z_0,\dots, Z_{n-1}
    \big) \ = \ Q_{n}^{*Z_{n-1}} \qquad \text{a.s.}
\end{equation} 
for every $n\geq 0$, where $q^{*z}$  is the $z$-fold convolution of the measure $q$.\\
As it turns out, probability generating functions (p.g.f.) are an important tool in the analysis of BPRE.
Thus let
$$f_n(s):=\sum_{k=0}^\infty s^k Q_n(k), \qquad (s\in [0,1])\ $$ 
be the probability generating function of the (random) offspring distribution $Q_n$. By $f$, we denote the generating function of $Q$. 
Throughout the paper, 
we denote  the conditioning on $Q_n$ indifferently by $\mathbb{E}[\cdot| Q_n]$ and $\mathbb{E}[\cdot| f_n]$. Also for
  the associated random environment we write both $\mathcal{E}=(f_1, f_2, ...)$ and $\mathcal{E}=(Q_1, Q_2, ...)$. 
In this notation, (\ref{transition}) can be written as
$$\E\big[s^{Z_{n}}\vert \mathcal{E}, \  Z_0,\dots,Z_{n-1}\big]=f_n(s)^{Z_{n-1}} \qquad \text{a.s.}\qquad (0\leq s\leq 1).$$

$ \qquad$ Another important tool in the analysis of BPRE is the \textbf{random walk associated with the environment} $(S_n \ : \ n\in\N_0)$. It determines many important properties, 
e.g. the asymptotics of the survival probability. $(S_n \ : \ n\in\N_0)$ is defined by 
$$S_0=0, \qquad S_{n}-S_{n-1}=X_n \quad (n\geq 1),$$ 
where $$X_n:=  \log m_{Q_n}=\log f_n'(1)$$
are i.i.d. copies of the logarithm of the mean number of offsprings  
$X:=\log (m_Q)=\log(f'(1))$. \\
The branching property then immediately yields
\begin{eqnarray}
\E[Z_n| Q_1,\ldots,Q_n, Z_0=1] &=& e^{S_n} \quad \mbox{a.s.}  \label{ew1}
\end{eqnarray}
The characterization of BPRE going back to \cite{athreya71} is classical:\\  
In the subcritical case ($\mathbb{E}[X]<0$), the population becomes extinct a.s. at an exponential rate. The same result is true in the critical case ($\mathbb{E}[X]=0$) (excluding the degenerated case
 when $\P_1(Z_1=1)=1$), but the rate of decrease of the survival probability is no longer exponential.  If $\mathbb{E}[X]>0$, the process survives with positive probability 
under quite general assumptions on the offspring distributions 
(see \cite{smith69}) and is called supercritical. Then   $\E[Z_1\log^+(Z_1) / f_1'(1)]<\infty$ ensures that the martingale $ e^{-S_n}Z_n$ has a positive finite limit 
on the non-extinction event:
$$\lim_{n\rightarrow \infty} e^{-S_n}Z_n =W, \qquad \P(W>0)=\P(\forall n  \in \N : Z_n>0)>0.$$ 
The large deviations are related to the speed of convergence of $\exp(-S_n)Z_n$ to $W$ and the tail of $W$. This latter is directly linked to 
the existence of moments and harmonic moments of  $W$. In the Galton Watson case, we refer to \cite{athreya} and \cite{Rouault}. For BPRE, Hambly \cite{Hambly} gives the  tail of $W$ in $0$, whereas  Huang \& Liu \cite{HuangLiu, HuangLiu2} obtain other various results in this direction.  \\

We establish here  an expression of the lower rate function for the large deviations of the BPRE, i.e.
 we specify the exponential rate of decrease of $\P(1\leq Z_n \leq e^{\theta n})$ for $0<\theta <\mathbb{E}[X]$. In the Galton Watson case, 
lower large deviations have been finely studied and the asymptotic probabilities are well-known, see e.g. \cite{FVlowerLDGW, FlWa, Ney}. In the case of a random environment, 
the rate function has been established in  \cite{bansaye08} when any individual leaves at least one offspring, i.e. $\P(Z_1=0)=0$. 
This result is extended here to the situation where $\P(Z_1=0)>0$  and the moment assumptions are relaxed.

We add that  for the problem of upper large deviations, the rate function has 
been established in \cite{BK09,BB10} and finer asymptotic results in the case of geometric offspring distributions can be found in  \cite{kozlov06,kozlov10}.  
Thus large deviations for BPRE become well understood, even if much work remains to get finer asymptotic results,  
deal with  weaker  assumptions or consider the B\"ottcher case ($\P(Z_1\geq 2)=1$).

\section{Preliminaries}
In the whole paper, we assume that $\mathbb{E}[X]>0$, i.e. the process is supercritical. Moreover, we are working in the whole paper under the following assumption. 

\begin{ann}\label{cramer}
 There exists an $s>0$ such that  $\E[e^{-sX}]<\infty$.  
\end{ann}
This  assumption ensures that a proper rate function $\Lambda$ of the random walk $(S_n:n\in\N)$ 
\begin{eqnarray}
 \Lambda(\theta) &:=& \sup_{\lambda\leq 0} \big\{ \lambda \theta -\log(\E[\exp(\lambda X)]) \big\}\  \label{rate}
\end{eqnarray}
exists. We note that the supremum is taken over $\lambda\leq 0$ and not over all $\lambda\in\mathbb{R}$. As we are only interested in lower deviations
here, this definition is more convenient as it implies $\Lambda(\theta)=0$ for all $\theta\geq \mathbb{E}[X]$.
We briefly recall some well-known facts about the rate function $\Lambda$ which are useful here (see \cite{dembo} for a classical reference on the matter). 
Define  $\phi(\lambda)=\log \E[\exp(\lambda X)]$, $\mathcal{D}_{\phi} = \{\lambda : \phi(\lambda) <\infty\}$ and let $\mathcal{D}^o_{\phi}$ 
be the interior of the set $\mathcal{D}_\phi$. Then the map 
$x \mapsto \Lambda(x)$ is strictly convex and infinitely often differentiable in the interior of the set 
$\{\theta \in\mathbb{R} : \theta=\phi'(\lambda) \text{ for some } \lambda \in \mathcal{D}^o_{\phi} \}$. Let $\theta=\phi'(\lambda_\theta)$ for some $\lambda_\theta\in  \mathcal{D}^o_{\phi}$.
 It then also holds that
\begin{align*}
 \Lambda'(\theta)= \lambda_\theta \ .
\end{align*}
Moreover for every $\theta\leq  \E(X)$
\be
\label{ldS2}
\lim_{n\rightarrow \infty} -\tfrac{1}{n}
\log \P( S_n \leq \theta n )  =
 \Lambda(\theta).
\ee

$\qquad$ In the following, we will denote 
\[\mathbb{P}(\cdot|Z_0=z)=\mathbb{P}_z(\cdot)\]
and write $\mathbb{P}(\cdot)$ when the initial population size is not relevant or can be taken equal to one. \\
To state the results, we will use the probability of staying positive but bounded which is treated in \cite{BB11}. Let us define 
$$\mathcal{I}:=\big\{ j \geq 1 \ :\ \P(Q(j)>0, Q(0)>0)>0\big\}$$
and introduce the set $Cl(\{z\})$ of integers that can be reached from $z\in\mathcal{I}$, i.e.
$$ Cl(\{z\}):=\big\{ k\geq 1 : \exists n \geq 0  \text{ with } \ \P_z(Z_n=k)>0\big\}. $$
In the same way, we introduce the set  $Cl(\mathcal{I})$  of integers which can be reached from $\mathcal{I}$ by  the process $Z$. More precisely,  
$$ Cl(\mathcal{I}):=\big\{ k\geq 1 : \exists n \geq 0  \text{ and } j\in\mathcal{I} \text{ with } \ \P_j(Z_n=k)>0\big\}. $$ 

We have
\begin{prop} 
\label{varrho}
(i) If $\P_1(Z_1=0)=0$ and $\mathbb{P}_1(Z_1=1)>0$, then for all $k$ and $j\in Cl(\{k\})$,
$$\lim_{n\rightarrow \infty} \tfrac{1}{n} \log \P_k( Z_n=j)= \log(\E(Q(1)^k)).$$
(ii) If $\E[X]>0$
 and $\P(Z_1=0)>0$ then the following limits exist,  coincide for all $k,j\in Cl(\mathcal{I})$ and belong to $[0,\infty)$,
$$\varrho:=\lim_{n\rightarrow \infty} \tfrac{1}{n} \log \P_k( Z_n=j)$$
\end{prop}
Note that $\mathbb{E}[X]>0$ implies that $\mathbb{P}(Z_1=1)<1$. The case $(i)$ can be proved directly for $j=k$ by observing 
that then $\{Z_n=k\}=\{Z_0=Z_1= \ldots=Z_n=k\}$ so $\P_k(Z_n=k)=\E(Q(1)^k)$. For the general case $j \in Cl(\{k\})$, the proof can be adapted from Lemma 7 in \cite{bansaye08}.

 The case $(ii)$ is proved in \cite {BB11}, Theorem 2.1.
In \cite{BB11}, some general conditions are stated which ensure $\varrho>0$ and $\varrho\leq \Lambda(0)$. It also gives a (non explicit) expression of $\varrho$ in terms of the successive differentiation 
of the p.g.f. $f_i$. \\
In the Galton Watson case, $f$ is constant, for every $i\geq 0$, $f_i=f$ a.s. Then, we recover 
the classical result \cite{AN} : 
$$\varrho=-\log f'(p_e), \qquad p_e:=\inf\{s \in [0,1] : f(s)=s\}.$$
Moreover, in the linear fractional case we have an explicit expression of $\varrho$.
We recall that a probability generating function of a random variable $R$ is linear fractional (LF) if there exist positive real numbers  $m$ and $b$ such that 
$$f(s)=1- \frac{1-s}{m^{-1}+  b m^{-2} (1-s)/2}\ ,$$
where $m=f'(1)$ and $b=f''(1)$. Then, we know from \cite{BB11} that under some conditions, which will be stated in the next section,
\begin{eqnarray}
\label{rateLF}
\varrho&=& \left\{ \begin{array}{l@{\quad,\quad}l}
                         -\log\mathbb{E}\big[e^{-X}\big]     & \mbox{if } \ \mathbb{E}[Xe^{-X}]\geq 0\\
\Lambda(0) & \mbox{else}
                             \end{array} \right. \ .
\end{eqnarray}

\section{Lower large deviations}\label{section3}
We introduce the following new rate function defined for $\theta,x\geq 0$ and any nonnegative function $H$
\begin{eqnarray*}
 \chi(\theta, x, H) &=& \inf_{t \in [0,1]} \big\{ tx +(1-t) H(\theta/(1-t))\big\}, \label{gamlower}
\end{eqnarray*}
with the convention $0\cdot\infty=0$.
\subsection{Main results}
To state the large deviation principle, we recall the definition of $\varrho$ and $\Lambda$ from the previous section and we need the following moment assumption:
\begin{ann}\label{as_strongly_supercrit} 
For every $\lambda>0$,
\begin{eqnarray*}
\label{Assumpt}
\E\Big[\left( \frac{f'(1)}{1-f(0)} \right)^{\lambda}\Big] <\infty.
\end{eqnarray*}
\end{ann}
Note that $\mathbb{P}(f(0)=1)=\mathbb{P}(Q(0)=1)>0$  would imply $\mathbb{P}(X=-\infty)>0$ which is excluded in the supercritical case. \\
We also denote $k_n \stackrel{subexp}{\longrightarrow} \infty$ when $k_n\rightarrow \infty$ but $k_n/\exp(\theta n)\rightarrow 0$ for every $\theta>0$, as $n\rightarrow \infty$.

\begin{theo}\label{thlower0} Under Assumptions \ref{cramer} \&  \ref{as_strongly_supercrit}  and  
$\mathbb{E}[Z_1\log^+(Z_1)/f_1'(1)]<\infty$ and $\mathbb{E}[|\log(1-f_1(0))|]<\infty$, the following assertions hold for every  $\theta\in \big(0, \mathbb{E}[X]\big]$.

(i) If $\P_1(Z_1=0)>0$, then for every $i\in Cl(\mathcal{I})$  
$$\lim_{n\rightarrow\infty} \tfrac{1}{n} \log \mathbb{P}_i(1\leq Z_n\leq e^{\theta n}) =-\chi(\theta,  \varrho,\Lambda ).$$
Moreover, $k_n \stackrel{subexp}{\longrightarrow} \infty$ ensures that 
$ 
\lim_{n\rightarrow\infty} \tfrac{1}{n} \log \mathbb{P}_i(1\leq Z_n\leq k_n)=-\varrho $.

(ii) If  $\P_1(Z_1=0)=0$,  then for every $i\geq1$,
$$\lim_{n\rightarrow\infty} \tfrac{1}{n} \log \mathbb{P}_i(1\leq Z_n\leq e^{\theta n}) =-\chi(\theta,  -\log \E[Q(1)^i], \Lambda ).$$
Moreover, $k_n \stackrel{subexp}{\longrightarrow} \infty$ ensures that
$\lim_{n\rightarrow\infty} \tfrac{1}{n} \log \mathbb{P}_i(1\leq Z_n\leq k_n)=\log \E[Q(1)^i]$.
\end{theo}
First, we note that $(ii)$  generalizes Theorem $1$ in \cite{bansaye08}, which required that both the mean and the variance of the reproduction laws were bounded (uniformly with respect to the environment). Moreover, $(i)$ provides an expression 
of the rate function in the more challenging case which allows extinction ($\P_1(Z_1=0)>0$).
We now try to extend this result and get rid of Assumption \ref{as_strongly_supercrit}, before discussing its interpretation and applying it to the linear fractional case. 
So we now consider 
\begin{ann}
\label{finvar} We assume that  $S$ is non-lattice, i.e. for every $r>0$, $\P( X \in r\Z)<1$.
Moreover, we assume that there exists a constant $0<d<\infty$ such that, 
\bea
 \label{as22} 
M_Q \leq d\cdot [ m_Q + (m_Q)^2]    \quad  \text{a.s.}, 
\nonumber 	  
\eea
where $M_q=\sum_{k\geq 0} k^2q(k)$ is the second moment of the probability measure $q$. \\
This condition is equivalent to the fact that $f''(1) / (f'(1)+f'(1)^2)$ is bounded a.s.
\end{ann}
This assumption does not require that $\mathbb{E}[f'(1)^\lambda]<\infty$ for every $\lambda>0$, contrarily   to Assumption \ref{as_strongly_supercrit}. But 
it  implies
that the standardized second moment of the offspring distributions is a.s. finite. It is e.g. fulfilled for geometric offspring distributions (see \cite{BK09}). We focus here on the case when subcritical environments may occur with positive probability, which implies in particular that $\P_1(Z_1=0)>0$. 
\begin{theo}\label{thlower} Under Assumption \ref{cramer} and $\P(X<0)>0$, for  any sequence $k_n \stackrel{subexp}{\longrightarrow} \infty$ and $i\in Cl(\mathcal{I})$, we have
$$\lim_{n\rightarrow\infty} \tfrac{1}{n} \log \mathbb{P}_i(1\leq Z_n \leq k_n)=-\varrho.$$
Under the additional Assumption \ref{finvar} and $\E[Z_1 \log^+(Z_1)]<\infty$, for every  $\theta\in \big(0, \mathbb{E}[X]\big]$,
\begin{eqnarray*}
&& \limsup_{n\rightarrow\infty} \tfrac{1}{n} \log \mathbb{P}_i(1\leq Z_n\leq e^{\theta n}) =-\chi(\theta, \varrho, \Lambda).
\nonumber
\end{eqnarray*}
\end{theo}
The proof of the upper-bound of this result  is very different from that of the previous theorem. It is deferred to Section \ref{upth2}. Let us now comment the large deviations results obtained by the two 
previous theorems. \\

\begin{figure}[here,t]
\setlength{\unitlength}{1cm}
\begin{center}
\includegraphics[angle=0,width=0.9\textwidth]{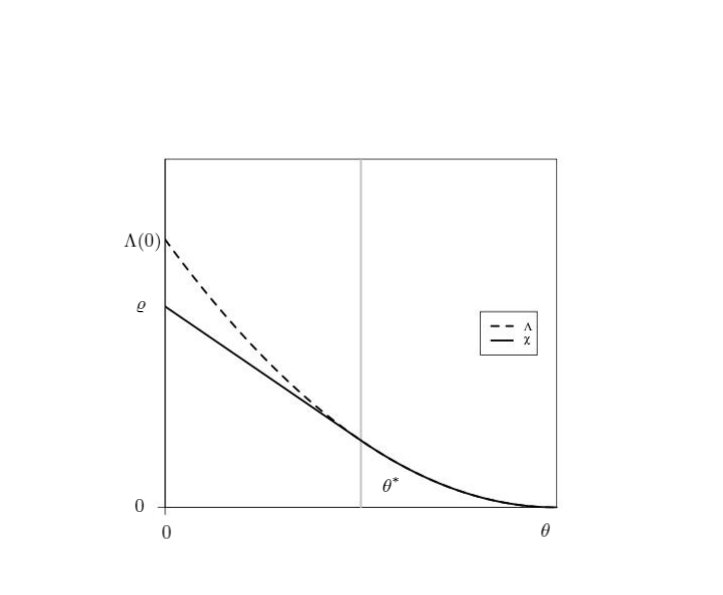}
\end{center}
\caption{\label{ratefunctionlow} $\chi$ and $\Lambda$ in the case $\theta^{\star}>0$.} 
\end{figure}

We note  that $\Lambda$ (and thus $\chi$) is a convex function which is continuous from below and thus has at most one discontinuity. 
If $\varrho< \Lambda(0)$, there is a phase transition of second order 
(i.e. there is a discontinuity of the second derivative of $\chi$). In particular, it occurs if $\Lambda(0)>-\log\mathbb{E}[Q(1)]$ since
we know from \cite{BB11} that $\varrho\leq -\log\mathbb{E}[Q(1)]$. In contrast to the upper deviations \cite{BK09, BB10}, there is no general description of this phase transition. It seems to heavily depend on the fine structure of the offspring distributions. In the linear fractional case, we are able to describe the phase transition more in detail (see forthcoming Corollary \ref{cor1}).\\

We also mention  the following representation of the rate function, whose proof follows exactly  Lemma 4 in \cite{BB10} and is left to the reader. We let $0\leq \theta^*\leq \mathbb{E}[X]$ be such that
\begin{align*}
 \frac{\varrho-\Lambda(\theta^*)}{\theta^*}= \inf_{0\leq \theta \leq \mathbb{E}[X]} \frac{\varrho-\Lambda(\theta)}{\theta}
\end{align*}
 Then, 
\begin{align*}
 \chi(\theta,\varrho,\Lambda)=\left\{\begin{array}{cc}
            \rho\big(1-\tfrac{\theta}{\theta^*}\big)+\tfrac{\theta}{\theta^*} \Lambda(\theta^*)  & \text{ if } \theta<\theta^*\\
      \Lambda(\theta) & \text{ if } \theta\geq \theta^*
             \end{array} \right.\ .
\end{align*}
\medskip\\

 We recall that $\varrho$ is known in the LF case from (\ref{rateLF}), 
and we derive the following result, which is proved in Section \ref{linearfrac}.
\begin{cor}\label{cor1} We assume that $f$ is a.s. linear fractional. Under Assumptions \ref{cramer} \& \ref{as_strongly_supercrit} or 
Assumptions \ref{cramer} \& \ref{finvar}, we have
 for all $\theta \in \big(0,\E[X]\big]$,
$$\lim_{n\rightarrow \infty}\tfrac{1}{n} \log \mathbb{P}_1(1\leq Z_n\leq e^{\theta n})
=\chi(\theta,\varrho,\Lambda)=\min\big\{- \theta-\log \mathbb{E}\big[e^{-X}\big],
\Lambda(\theta)\big\}.$$ 
More explicitly, $\theta^*=\E\big[X\exp(-X)\big]/\E[\exp(-X)]$. \\
If $\theta<\theta^*$, then $\chi(\theta,\varrho,\Lambda)=- \theta-\log \mathbb{E}\big[e^{-X}\big]$, otherwise $\chi(\theta,\varrho,\Lambda)
=\Lambda(\theta)$.
\end{cor}

We note that if the offspring-distributions are geometric, Assumption \ref{finvar} is automatically fulfilled (see \cite{BK09}). Moreover, except for the 
degenerated case $\mathbb{P}(Z_1=0)=1$, we have $\mathbb{P}(Z_1=1)>0$ in the linear fractional case.  
Note that the non-lattice assumption made in Assumption \ref{finvar} can be dropped since one can directly proved in the LF case that $\varrho \leq \Lambda(0)$.
Finally, starting from $k\geq 1$ individuals, the result holds if $\varrho$ is replaced by $\varrho_k$, where  $\varrho_k=\varrho$ if $\P_1(Z_1=0)>0$ and  $\varrho_k=-\log(\E(Q(1)^k))$ if $\P_1(Z_1=0)=0$. 

\subsection{Interpretation}

\begin{figure}[here]
\setlength{\unitlength}{1cm}
\begin{center}
\includegraphics[angle=0,width=0.5\textwidth]{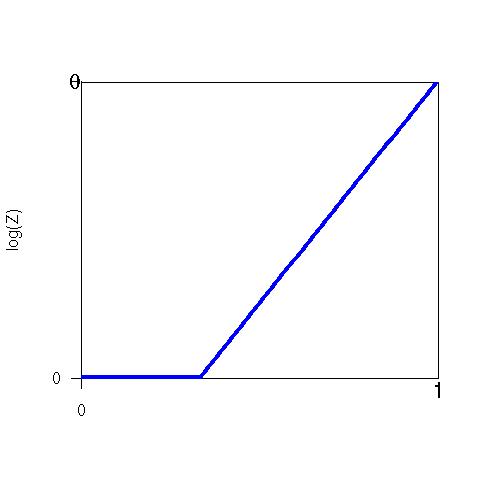}
\end{center}
\caption{\label{pathlow} Most probable path for the event $\{1\leq Z_n\leq e^{\theta n}\}$ with $0<\theta<\theta^\star$.} 
\end{figure}

Let us explain the  rate function  $\chi$ and
describe  the large deviation event $\{1\leq Z_n\leq e^{\theta n}\}$ for some $0<\theta<\mathbb{E}[X]$ and $n$ large. 
This corresponds to observing a population in generation $n$ which is much smaller than expected, but still alive. A possible path that led to this event looks as follows 
(see Figure \ref{pathlow}).\\
 During a first period, until generation $\lfloor tn\rfloor$ ($0\leq t\leq 1$), the population stays small but alive, despite the fact that the process is supercritical.
 The probability of such an event is exponentially small and  of order $\exp(-\varrho\lfloor nt\rfloor+o(n))$. 
Later, the population grows in a supercritical environment but less favorable than the typical one, i.e. $\{S_n-S_{\lfloor nt\rfloor} \leq \theta n\}$. 
This atypical environment sequence has also exponentially small probability, of order $\exp(-\Lambda(\theta/(1-t))\lfloor n(1-t)\rfloor+o(n))$.
The probability of the large deviation event then results from maximizing the product of these two probabilities. \medskip\\

More precisely, we may follow \cite{bansaye08} to  check that 
 the infimum of $\chi$ is reached at a unique
point $t_\theta$ by convexity arguments. Thus
$$\chi(\theta) = t_\theta \varrho + (1-t_\theta) \Lambda(\theta/(1-t_\theta)), \qquad t_\theta \in[0,1-\theta/\E[X]]$$ 
and we can  define the function
$f_\theta : [0,1] \mapsto \R_+$ for each $\theta < \E[X]$ as follows
$$f_\theta(t):= \left\{\begin{array}{ll}  0, \qquad &  \t{if} \ t< t_\theta \\
 \frac{\theta}{1-t_\theta}(t-t_\theta), \quad & \t{if} \ t\geq t_\theta.
\end{array}
\right.
$$
Then, conditionally on $\{1\leq Z_n \leq \exp(n\theta)\}$,  the process  $(\log(Z_{[tn]})/n : t\in [0,1])$ converges in finite dimensional distributions 
to the function  $(f_\theta(t) : t\in[0,1])$.   \\

From the point of view of theoretical ecology, these results shed light on the environmental and demographical stochasticity of the model.
More precisely, randomness in a BPRE comes  both from  the random evolution of the environment (environmental stochasticity) and the random reproduction of each 
individual (demographical stochasticity). Thus a rare event $\{1\leq Z_n \leq \exp(n\theta)\}$ for $n$ large and $\theta <\E[X]$ may be due to a rare sequence of 
environments (less favorable than usual since $Z_n \leq \exp(n\theta)$, but not bad enough to provoke extinction) and/or to unsual reproductions of individuals. 
Our results show that it is a non-trivial combination of both.  \\
In a first period $[0,t_{\theta}]$, the population just survives thanks to a combination of environmental and demographical stochasticity 
(we call this period \textsl{survival period}). 
If $\P(Z_1=0)=0$, we know that the population remains constant.  Thus the typical environment $f$ is biased by $\P_1(Z_1=1 \vert f)=f'(0)$ and the number of offspring is 
forced to be $1$ for (almost) all individuals. If $\P(Z_1=0)>0$ and $\varrho<\Lambda(0)$, e.g. in the LF case, 
again it is a combination of the demographical and environmental stochasticity. If $\P(Z_1=0)>0$ and $\varrho=\Lambda(0),$ the time of the survival period is reduced to $0$ : $t_{\theta}=0$. \\
In a second period $[t_{\theta},1]$, the population grows exponentially but at a lesser rate than usual. This is only due to the environmental stochasticity : the typical environment $f$ is not biased by the
 mean offspring number $f'(1)$.

\subsection{Application to Kimmel's model : cell division with parasite infection}
As an illustration and a motivation we deal with  the following
branching model for cell division with parasite infection. It is described and studied in  \cite{Ban08,Ban10}.
In each generation, the cells give birth to two daughter cells and
the cell population is the binary tree. The model takes into 
account unequal sharing of parasites in the two daughter cells, following experiments made
in Tamara's Laboratory in Hopital Necker (Paris). 

More explicitly, we assume that the parasites reproduce following a Galton-Watson process with reproduction law $(p_k : k\geq 0)$.
We consider a random variable $P\in (0,1)$ a.s. and, for convenience, we assume that its distribution is symmetric with respect to $1/2$ : $P \stackrel{d}{=} 1-P$.
This random parameter gives the binomial repartition of the parasites in each daughter cell. It is picked in an i.i.d manner for each cell. 
Thus, conditionally on the fact that the cells contain $k$ parasites when it divides and  conditionally on this parameter being equal to $p$, the number of parasites inherited by the first daughter cell follows a binomial distribution with parameters $(k,p)$, whereas the other parasites go in the other daughter cell. In other words, each parasite is picked independently into the first daughter cell with probability $p$.

The number of cells in generation $n$ is $2^n$.
Then, a simple computation proves that the number of cells $N_n[a,b]$ in generation $n$ whose number of parasites is between $a$ and $b$ satisfies
$$\E\big[N_n[a,b]\big]=2^n\P(Z_n \in [a,b]),$$
where $Z_n$ is a BPRE whose environment is given by the random variable (r.v.) $P$ : 
$$\P_1( Z_1 = i   \ \vert \ P=p) = \sum_{k\geq i}^{\infty} p_k p^i(1-p)^{k-i}.$$
As a consequence of the previous Theorems, we can   derive the mean behavior of the number of cells infected by a positive number of parasites which is smaller than usual:
$$\frac{1}{n} \log \E\big[N_n[1,\exp(n\theta)]\big]=\log(2) -\chi(\theta, \varrho, \Lambda) \qquad \theta <\E(X),$$
where $\Lambda$ is the Fenchel Legendre transform of the r.v. 
$$X:=\log (\sum_{k\geq 0} kp_k) +\log(P)$$ 
and $\varrho$ is inherited  from Proposition 2.1. $(i)$ when $p_0>0$.
In particular, let us assume that $(p_k)_{k\geq 0}$ is a linear fractional offspring distribution, i.e. there exist
$a \in [0,1]$ and $q\in [0,1)$ such that
$$p_0=a, \qquad p_k= (1-a)(1-q)q^k (k\geq 1).$$
 Then 
$$\P_1( Z_1 = i   \ \vert \ P=p) = \sum_{k\geq i}^{\infty} a  q^k p^i(1-p)^{k-i}= \frac{a}{1-(1-p)q} (q  p)^i,$$
i.e. the offspring distribution for the branching process $Z$ is also a.s. linear fractional. Thus we can apply Corollary  \ref{cor1} and $\varrho$ can be 
calculated explicitly from the distribution of $P$. 
Furthermore, solving $\chi(\theta, \varrho, \Lambda)>\log 2$ yields the set of $\theta$ such that  we  observe cells infected by a positive number but less than $\exp(n\theta)$ parasites 
(for large times).

\section{Proof of lower large deviations}\label{section7}
First, we focus on the lower bound, which is easier and can be made under general assumptions (satisfied in both Theorems  \ref{thlower0} and \ref{thlower}).
We split then the proof of the upper bounds in two parts, working with Assumption \ref{as_strongly_supercrit}  in the first one, and then with $\P(X<0)>0$ and Assumption \ref{finvar} in the second. 
Finally, we prove the theorems combining these results.

\subsection{Proof of the lower bound for Theorems  \ref{thlower0} and \ref{thlower} }
\label{proofoflower}
First we note that, if  the associated random walk has exceptional values, the same is true for the branching process $Z$. 
The estimation $Z_n\approx \E[Z_n \ \vert  \mathcal{E}]=\exp(S_n)$ a.s. gives a lower bound in the following way.
 If  $\E[Z_1\log^+(Z_1)]<\infty$, we know from \cite{athreya71} that the limit of the martingale $Z_n\exp(-S_n)$ is non-degenerated. Then a direct generalization 
of \cite[Proposition 1]{bansaye08}
ensures that
\begin{eqnarray}
 \liminf_{n\rightarrow\infty} \tfrac{1}{n} \log \mathbb{P}_j(1\leq Z_n \leq e^{\theta n}| S_n\leq (\theta+\varepsilon) n) &=& 0, \nonumber 
\end{eqnarray}
for all $j\geq 1$ and  $\varepsilon>0$. It relies on the same change of measure as in the proof of \cite[Proposition 1]{bansaye08}:
$$\w{\P} ( Q  \in  \d q) :=  \frac{ m(q)^{\lambda_c}}{\E\big[ m(Q)^{\lambda_c}\big]} \P(Q \in \d p),$$
where $\lambda_c$ is the argmax of $\lambda \rightarrow \lambda c- \varphi(\lambda)$:
$$\sup_{\lambda\leq 0}\{\lambda c-\varphi(\lambda)\}=\lambda_c c- \varphi(\lambda_c).$$ 
As $\Lambda$ is non-increasing, continuous from below and convex and thus a right-continuous function,  $\Lambda(\theta+\varepsilon)\rightarrow\Lambda(\theta)$ as $\varepsilon\rightarrow 0$. Then, for every $0<\theta<\mathbb{E}[X]$ such that $\Lambda(\theta)<\infty$, we have
\begin{eqnarray}\label{lobo}
 \liminf_{n\rightarrow\infty} \tfrac{1}{n} \log \mathbb{P}_j(1\leq Z_n \leq e^{\theta n}| S_n\leq \theta n)=0. 
\end{eqnarray}
Now we can prove the following  result

\begin{lemma} \label{proplow} Let $z\geq 1$. We assume that  $\E[Z_1\log^+(Z_1)]<\infty$ and  that 
$$\varrho_z=-\lim_{n \rightarrow \infty} \tfrac{1}{n} \log\P_z(1\leq Z_{n}\leq b)$$ 
exists and does not depend on $b$ large enough.
Then   for every $\theta \in \big(0, \mathbb{E}[X]\big]$, we have
$$ \liminf_{n\rightarrow\infty} \tfrac{1}{n} \log \mathbb{P}_z(1\leq Z_n\leq e^{\theta n}) \geq -\chi(\theta, \varrho_z,\Lambda). $$
\end{lemma}
\begin{proof}
 We  decompose the probability following a time $t\in [0,1)$ when the process goes beyond $b$. Using the large deviations principle  satisfied by the random walk $S$, 
we have for every $\varepsilon>0$ and $n$ large enough
\begin{align}
 & \mathbb{P}_z(1\leq Z_n\leq e^{\theta n}) \nonumber \\ 
& \geq \mathbb{P}_z(1\leq Z_{\lfloor t n\rfloor}\leq b)  
\min_{1\leq k \leq b} \mathbb{P}_k(1\leq Z_{\lfloor (1-t)n\rfloor} \leq e^{\theta n}; S_{\lfloor (1-t)n\rfloor }\leq e^{\theta n})\nonumber \\
& \geq  \mathbb{P}_z(1\leq Z_{\lfloor t n\rfloor}\leq b) e^{-\Lambda\left(\tfrac \theta{1-t}\right) n(1-t)+\varepsilon n} \min_{1\leq k \leq b} \mathbb{P}_k\Big(1\leq Z_{\lfloor(1- t)n\rfloor} 
\leq  e^{\tfrac \theta{1-t} n(1-t)}\Big|S_n \leq e^{\tfrac \theta{1- t} n(1-t)}\Big). \nonumber
\end{align}
 Note that the above inequality is trivially fulfilled if $\Lambda(\theta)=\infty$. The definition of $\varrho_z$ and  (\ref{lobo}) yield with $b$ large enough 
and for every $\varepsilon>0$
$$ \liminf_{n\rightarrow\infty} \tfrac{1}{n} \log \mathbb{P}_z(1\leq Z_n\leq e^{\theta n}) 
\geq -\inf_{t\in [0,1)} \big\{ t\varrho_z+(1-t)\Lambda\big(\theta/(1-t)\big)+\varepsilon\big\}.$$
Adding that 
$\mathbb{P}_z(1\leq Z_n\leq e^{\theta n}) \geq \P_z(1\leq Z_{n}\leq b)$ for $n$ large enough, we can take the latter infimum of $[0,1]$, 
again with the convention $0\cdot\infty=0$. 
Taking the limit $\varepsilon\rightarrow 0$ yields the expected lower bound  $-\chi(\theta, \varrho_z,\Lambda)$.
\end{proof}

\subsection{Proof of the upper bound for Theorem  \ref{thlower0} (i) and (ii)}
The next lemma ensures that a large population typically grows as its expectation and thus follows the random walk of the environment $S$. The start of the proof of this proposition is in the same vein as \cite{bansaye08}, but the situation is much more involved since $\P_1(Z_1 =0)$ may be positive,  $f'(1)$ may not be bounded a.s. and the variance of the reproduction laws may be infinite with positive probability. 
\begin{lemma} \label{prop_help} 
Under Assumption \ref{as_strongly_supercrit}, for every $\varepsilon>0$ and for every $a>0$, there exist constants $c,b\geq 1$ such that for every $n\in\N$
$$\sup_{z\geq b}\P_z(Z_n\leq e^{S_n-n\varepsilon} ; Z_1 \geq b, ..., Z_n\geq b)
\leq c\ e^{-an}.$$
\end{lemma}
\begin{proof}
Let us introduce the ratio of the successive sizes of the population
\begin{align*}
R_i:=Z_{i}/Z_{i-1}, \qquad i\in \{1,\ldots,n\}.
\end{align*}
Recalling that $\log f'_i(1)=X_i$, we can rewrite 
\begin{align*}
\frac{e^{S_n-n\varepsilon}}{Z_n} = Z^{-1}_0 \prod_{i=1}^n \frac{f'_i(1) }{e^{\varepsilon}R_i}.
\end{align*}
Then for every $\lambda>0$, we can use the classical Markov inequality  $\mathbb{P}(Y\geq 1) \leq \mathbb{E}[Y^\lambda]$ for any nonnegative random variable $Y$ and get for every $z\geq b$ 
 \begin{align*}
\P_{z}(Z_n\leq e^{S_n-n\varepsilon} &;  Z_1 \geq b, ..., Z_n\geq b)\\
& \leq b^{-\lambda} \E\Big[\prod_{i=1}^n (f_i'(1)/(e^{\varepsilon}R_i))^{\lambda} ;  Z_1 \geq b, ..., Z_n\geq b \Big]\nonumber \\
& = b^{-\lambda} \E\Big[\prod_{i=1}^n (e^{\varepsilon}R_i/f_i'(1))^{-\lambda} ;  Z_1 \geq b, ..., Z_n\geq b \Big]\ .
\end{align*}
Now we introduce the following random variable 
$$M_{\lambda}(b,g):=\sup_{k \geq b} \E\Big[ \Big(e^{\varepsilon}\frac{\sum_{i=1}^k N_i^g}{k g'(1)}\Big)^{-\lambda};\sum_{i=1}^k N_i^g>0\Big]$$
where  $N^g_i$ are i.i.d., integer valued random variables with (fixed) p.g.f. $g$. By the branching property, we may write a.s.
\begin{align*}
M_{\lambda}(b,f_i)&=\sup_{k \geq b} \E\Big[ \Big(e^{\varepsilon}\tfrac{Z_{i+1}}{Z_i f_i'(1)}\Big)^{-\lambda}  ; Z_{i+1} > 0 \ \Big| \ f_i ,Z_i=k  \Big]\\
&=\sup_{k \geq b} \E\Big[ \Big(e^{\varepsilon}\tfrac{R_{i+1}}{f_i'(1)}\Big)^{-\lambda}  ; Z_{i+1} > 0 \ \Big| \ f_i ,Z_i=k  \Big]. 
\end{align*}
hen, by conditioning on the successive sizes of the population, we obtain
\begin{align*}
\P_{b}(Z_n&\leq e^{S_n-n\varepsilon} ;  Z_1 \geq b, ..., Z_n\geq b)\\
& \leq b^{-\lambda}\E\bigg[\prod_{i=1}^{n-1} \Big(\tfrac{e^{\varepsilon}R_i}{f_i'(1)}\Big)^{-\lambda}  \E\Big[ \Big(\tfrac{e^{\varepsilon}R_n}{f_n'(1)}\Big)^{-\lambda} ; Z_n\geq b  \ \vert \ f_{n} , Z_{n-1} \Big] ;  Z_1 \geq b, ..., Z_{n-1}\geq b\bigg] \\
& \leq b^{-\lambda} \E\Big[\prod_{i=1}^n M_{\lambda}(b,f_i)\Big] \\
& = b^{-\lambda}\E\big[ M_{\lambda}(b,f)\big]^n.
\end{align*}
$\qquad$ We now want to prove that for every $\alpha \in (0,1)$, there exist $\lambda, b>0$ such that $\E[ M_{\lambda}(b,f)] \leq \alpha$. 
Let $g$ be fixed and deterministic. The idea is that for every $g$, $\sum_{i=1}^k N_i^g/k \rightarrow  g'(1)$ a.s. as $k\rightarrow \infty$ by the law of large numbers. We will be able to derive that 
$$\E\Big[ \Big(e^{\varepsilon}\frac{\sum_{i=1}^k N_i^g}{kg'(1)}\Big)^{-\lambda} ; \sum_{i=1}^k N_i^g>0\Big]\rightarrow e^{-\lambda \varepsilon},$$
as $k\rightarrow \infty$ and   $M_{\lambda}(b,f)\rightarrow e^{-\lambda \varepsilon}$ a.s. as $b$ goes to infinity. Under suitable conditions, we  are then able to prove that  $\E [M_{\lambda}(b,f)]\rightarrow e^{-\lambda \varepsilon}$. Finally, considering $\lambda>0$ such that $e^{-\lambda \varepsilon} < e^{-a}$ and  $b$ large enough gives us the result.  \medskip\\
Let us now present the details of the proof. First we fix a p.g.f. $g$ such that $g(0)<1$ and $E[N_1^g]=g'(1)<\infty$. Then the law of large numbers ensures
$$Y_k:=\Big(e^{\varepsilon}\frac{\sum_{i=1}^k N_i^g}{k g'(1)}\Big)^{-\lambda} \stackrel{k\rightarrow \infty}{\longrightarrow} e^{-\lambda \varepsilon} \qquad \P \ \text{-- a.s. }$$
Moreover $\sum_{i=1}^k N_i^g$ is stochastically larger than a random variable $B(k,g)$ with binomial distribution of parameters $(k,1-g(0))$. 
Applying the classical large deviations upper bound for Bernoulli random variables (see e.g. \cite{dembo, hollander}) yields for $x\geq 0$
\begin{align*}
\P\Big(Y_k \geq x ; \sum_{i=1}^k N_i^g>0 \Big)&\leq  \P\big(B(k,g)\leq k\  x^{-1/\lambda} g'(1)e^{-\varepsilon}\big) \leq \exp\big(-k\ \psi_g(x^{-1/\lambda} g'(1)e^{-\varepsilon})\big), 
\end{align*} where the 
function $\psi_g(z)$ is zero if $z\geq 1-g(0)$ and 
 positive for $0\leq z< 1-g(0)$.  It is specified by  the Fenchel Legendre transform of a Bernoulli distribution, i.e. for $0\leq z\leq 1-g(0)$,
$$\psi_g(z) = z\log\big(\tfrac{z}{1-g(0)}\big)+(1-z)\log\big(\tfrac{1-z}{g(0)}\big).$$
Moreover $\{ \sum_{i=1}^k N_i^g>0 \} \subset \{Y_k \leq k^{\lambda}d\}$ with $d=(g'(1)e^{-\varepsilon})^{\lambda}$. Thus
$$\E\Big[Y_k \ind_{Y_k\geq x} ; \sum_{i=1}^k N_i^g>0 \Big]\leq d k^{\lambda} \P\Big(Y_k \geq x ;\sum_{i=1}^k N_i^g>0\Big)  \leq dk^{\lambda}\exp\big(-k\psi_g(x^{-1/\lambda} g'(1)e^{-\varepsilon})\big) .$$ 
Let us choose $x$ large enough such that $\psi_g(x^{-1/\lambda} g'(1)e^{-\varepsilon})>0$. Then letting $k\rightarrow\infty$, the right-hand side of the above 
equation converges to 0.  Moreover, we can apply the bounded convergence theorem to $Y_k\ind_{Y_k\leq x, \sum_{i=1}^k N_i^g>0}$ to get
$$ \limsup_{k\rightarrow \infty} \E\Big[ Y_k ;  \sum_{i=1}^k N_i^g>0 \Big] 
= \E\Big[ \limsup_{k\rightarrow \infty}\big( Y_k\ind_{Y_k\leq x, \sum_{i=1}^k N_i^g>0}\big)\Big]\leq e^{-\lambda \varepsilon}.$$
Recalling that  $M_{\lambda}(b,g)$ decreases with respect to $b$, we get for every $g$
$$ \lim_{b\rightarrow \infty}  M_{\lambda}(b,g) \leq  e^{-\lambda \varepsilon}\ .$$

$\qquad$ Second, we apply the bounded convergence theorem again and finish the proof by integrating the previous result with respect to the environment. To check  that
$$\E[M_{\lambda}(1,f)]<\infty,$$
we define for any p.g.f. $g$ with  $g(0)<1$  and $g'(1)<\infty$ the real numbers
$$x_g:=\left(e^{-\varepsilon}\frac{2g'(1)}{1-g(0)} \right)^{\lambda}, \qquad    y_g:=(ke^{-\varepsilon}g'(1))^{\lambda}.$$
For $k$ large enough, we have $x_g<y_g$. We also note that $x\geq x_g$ implies  that  $x^{-1/\lambda} g'(1)e^{-\varepsilon}\leq (1-g(0))/2$. Moreover, $\sum_{i=1}^k N_i^g>0$ implies $Y_k\leq y_g$, and thus   
\begin{align*}
\E\Big[Y_k   ; \sum_{i=1}^k N_i^g>0\Big] &= \int_0^{y_g} \P\Big( Y_k\geq x ; \sum_{i=1}^k N_i^g>0 \Big) dx \\
& \leq x_g+ \int_{x_g}^{dk^{\lambda}} \exp\big(-k\psi_g(x^{-1/\lambda} g'(1)e^{-\varepsilon})\big) dx  \\
& \leq x_g+ dk^{\lambda} \exp\big(-k\psi_g\big(\tfrac{1-g(0)}{2}\big)\big).
\end{align*}
Now we maximize the right-hand side with respect to $k\geq 1$. 
Using that for all $\alpha>0, x\geq 0$,  $x^{\lambda}e^{-\alpha x}\leq (\lambda/\alpha)^{\lambda}e^{-\lambda}$ and recalling that $d=(g'(1)e^{-\varepsilon})^{\lambda}$, we get 
\begin{align}
M_{\lambda}(1,g)= \sup_{k\geq 1}\E\Big[Y_k   ; \sum_{i=1}^k N_i^g>0\Big] \leq x_g+(e^{-\varepsilon}g'(1))^{\lambda}\lambda^{\lambda}e^{-\lambda}\psi_g\big(\tfrac{1-g(0)}{2}\big)^{-\lambda}.\label{sup1}
\end{align}
Finally, we observe that $\psi_g(z)$ is a nonnegative convex function which reaches $0$ in $1-g(0)$. Thus $0\leq x\leq y  \leq 1-g(0)$ implies
$\psi_g(x) \geq (x-y)\psi_g'(y)$ and in particular
$$\psi_g\big( \tfrac{1-g(0)}{2}\big) \geq - \tfrac{1-g(0)}{4}\psi_g'\big(\tfrac{1-g(0)}{4}\big).$$
As $\psi_g'(z)=\log(\tfrac{zg(0)}{(1-z)(1-g(0))}\big)$ and $\log(1-x)\leq x$ for $x>0$,
we get that
\begin{align}
\psi_g\big( \tfrac{1-g(0)}{2}\big) \geq -\tfrac{1-g(0)}{4} \log\big(1-\tfrac{3}{3+g(0)}\big)\geq  \frac 34 \frac{1-g(0)}{3+g(0)}\geq \frac {3(1-g(0))}{16}. \label{sup2}
\end{align}
Combining the inequalities (\ref{sup1}) and (\ref{sup2}) yields
$$M_{\lambda}(1,f) \leq a(\varepsilon,\lambda) \left( \frac{f'(1)}{1-f(0)} \right)^{\lambda} \qquad \text{a.s.},$$
where $a(\varepsilon,\lambda)$ is a finite positive constant, only depending on $\varepsilon$ and $\lambda$.
Thus Assumption \ref{as_strongly_supercrit}  ensures that $\E[M_{\lambda}(1,f)]<\infty$. Adding that $M_{\lambda}(b,f)\leq M_{\lambda}(1,f)$ a.s. for $b\geq 1$, we apply the bounded convergence theorem to obtain 
\begin{align*}
\lim_{b\rightarrow\infty}  \mathbb{E}\big[M_{\lambda}(b,f)\big] = \mathbb{E}\big[\lim_{b\rightarrow\infty}  M_{\lambda}(b,f)\big] \leq e^{-\lambda \varepsilon}.
\end{align*}
Then,  choosing $b$ large enough,
\begin{align*}
  \mathbb{E}\big[  M_{\lambda}(b,f)\big] \leq 2e^{-\lambda \varepsilon}. 
\end{align*}
Letting $\lambda$ such that $2e^{-\lambda \varepsilon}\leq e^{-a}$  ends up the proof.  
\end{proof}

\begin{lemma} \label{propupp} Let $z\geq 1$ and  assume that 
$$\varrho_z=-\lim_{n \rightarrow \infty} \tfrac{1}{n} \log\P_z(1\leq Z_{n}\leq b)$$ 
exists and does not depend on $b$ large enough. Then, under Assumption \ref{as_strongly_supercrit}, for every $\theta \in \big(0,\E[X]\big]$,
$$\limsup_{n \rightarrow \infty} \tfrac{1}{n} \log \P_z\big(1\leq Z_{n}\leq \exp(n\theta)\big)\leq -\chi(\theta, \varrho_z, \Lambda).$$ 
\end{lemma}
\begin{proof}
We define the last moment when the process is below $b$ before time $n$ :
$$\sigma_b=\inf\{ i<n : Z_{i+1} \geq b , \cdots, Z_n\geq b\}, \qquad (\inf \varnothing =\infty).$$
Let $\theta>0$. Then 
summing over $i$ leads to 
\begin{align*}
&\P_{z}(1 \leq Z_n \leq e^{\theta n}) \\
&\qquad \leq  \sum_{i=0}^{n-1} \P_k( 1\leq Z_{n}\leq e^{\theta n}, \sigma_b=i) + \P(1\leq Z_n\leq b)\\
&\qquad \leq \sum_{i=0}^{n-1} \P_{z}(1\leq Z_{i}\leq b) \sup_{j\geq b}\P_j( 1\leq Z_{n-i-1}\leq e^{\theta n},  Z_1\geq b, ..., Z_{n-i-1}\geq b)
 + \P(1\leq Z_n\leq b) \\
&\qquad \leq \P(1\leq Z_n\leq b)+ \sum_{i=0}^{n-1} \P_{z}(1\leq Z_{i}\leq b)\Big[
 \P(S_{n-i-1} \leq  \theta n +n \varepsilon ) \\
&\qquad \qquad \qquad +\sup_{j\geq b}\P_{j}(Z_{n-i-1}\leq e^{\theta n}, \ S_{n-i-1}>  \theta n +n \varepsilon, Z_1\geq b, ..., Z_{n-i-1}\geq b )\Big]\\
&\qquad \leq n \sup_{t\in [0,1]} \Big\{\mathbb{P}_z(1\leq Z_{\lfloor nt\rfloor}\leq b) \Big[\P(S_{n-\lfloor nt\rfloor-1} \leq  \theta n +n \varepsilon ) \\
& \qquad \qquad \qquad +\sup_{j\geq b}\P_{j}(Z_{n-\lfloor nt\rfloor-1}\leq e^{\theta n}, \ S_{n-\lfloor nt\rfloor-1}>  \theta n +n \varepsilon, Z_1\geq b, ..., Z_{n-\lfloor nt\rfloor-1}\geq b )\Big]\Big\}.
\end{align*}
As the limit and the supremum can be exchanged, we get that
\begin{align*}
 &\limsup_{n\rightarrow\infty} \tfrac 1n \log\P_{z}(1 \leq Z_n \leq e^{\theta n}) \leq \\
&\qquad \sup_{t\in [0,1]}\Big\{ \limsup_{n\rightarrow\infty} \tfrac 1n \log \mathbb{P}_z(1\leq Z_{\lfloor nt\rfloor}\leq b)
+\limsup_{n\rightarrow\infty} \tfrac 1n \log\Big[\P(S_{n-\lfloor nt\rfloor-1} \leq  \theta n +n \varepsilon )\\
&\qquad \qquad +\sup_{j\geq b}\P_{j}(Z_{n-\lfloor nt\rfloor-1}\leq e^{\theta n}, \ S_{n-\lfloor nt\rfloor-1}>  \theta n +n \varepsilon, Z_1\geq b, ..., Z_{n-\lfloor nt\rfloor-1}\geq b )\Big]\Big\}.
\end{align*}
For the first summand in the supremum,  by assumption, we have for every $t\in [0,1]$, 
$$\lim_{n \rightarrow \infty} \tfrac{1}{n} \log \P_z(1\leq Z_{\lfloor tn\rfloor}\leq b)=-t\varrho_z.$$ 
For the first probability in the second summand, we use the classical large deviation inequality for the random walk $S$ (see (\ref{ldS2})) to get for every $t\in [0,1]$ that for 
$\theta\in \big(0,\mathbb{E}[X]\big]$ and $\varepsilon>0$ small enough
$$\limsup_{n\rightarrow\infty}\tfrac 1n \log \P(S_{\lfloor (1-t)n\rfloor} \leq  \theta n +n \varepsilon)=-(1-t)\Lambda\big(\tfrac{\theta+\varepsilon}{1-t}\big).$$
with  the convention $0\cdot\infty=0$.
For the last probability, we apply Lemma \ref{prop_help}, which prevents a large population form deviating from the random environment. More precisely,  
for every $\varepsilon>0$, we can choose $b$ large enough such that $\sup_{j\geq b}\P_{j}(Z_{n-i-1}\leq e^{\theta n}, \ S_{n-i-1}\geq  \theta n +n 
\varepsilon, Z_1\geq b, ..., Z_{n-i-1}\geq b )$ decreases 
faster than $\exp(-\varrho_z [n-i-1])$ as $n$ goes to infinity. Thus, for $b$ large enough and every $t\in [0,1]$,
$$\limsup_{n \rightarrow \infty} \tfrac{1}{n} \log \sup_{j\geq b}\P_{j}(Z_{n(1-t)}\leq e^{\theta n}, \ S_{n(1-t)}>  \theta n +n \varepsilon, Z_1\geq b, ..., Z_{n(1-t)}\geq b ) \leq -\varrho_z(1-t).$$
Combining these upper bounds yields
$$\limsup_{n \rightarrow \infty} \tfrac{1}{n} \log \P_z\big(1\leq Z_{n}\leq \exp(n\theta)\big)\leq -\min\big\{
\inf_{t\in [0,1)}\big\{ t\varrho_z +(1-t) \Lambda\big(\tfrac{\theta+\varepsilon}{1-t}\big)\big\} ; \varrho_z \big\}$$
Letting $\epsilon \rightarrow 0$, by right-continuity of $\Lambda$, the right-hand side goes to
$$\inf_{t\in [0,1]}\big\{ t\varrho_z +(1-t) \Lambda\big(\tfrac{\theta}{1-t}\big)\big\}=\chi(\theta, \varrho_z,\Lambda),$$
with the convention $0\cdot\infty=0$.
It completes the proof.
\end{proof}

\subsection{Proof of the upper bound for Theorem \ref{thlower}}
\label{upth2}
We assume  here that subcritical environments occur with a positive probability. First,  we consider the probability of having less than exponentially 
many individuals in generation $n$ and prove that the decrease of this probability is still given  by $\varrho$. 
We derive the upper bound of the second part of the theorem using Assumption \ref{Assumpt} and an additional lemma.
\begin{lemma}\label{l_rho2}
If   $\mathbb{P}(X<0)>0$, then  for every $z\in Cl(\mathcal{I})$, 
\begin{align*}
\varrho&=\lim_{n\rightarrow\infty} \tfrac{1}{n} \log\mathbb{P}_{z}(Z_{n}=z) 
= \lim_{\theta\rightarrow 0}\liminf_{n\rightarrow\infty}\tfrac{1}{n} \log\mathbb{P}_{z}(1\leq Z_n\leq e^{\theta n})\\
&=\lim_{\theta\rightarrow 0}\limsup_{n\rightarrow\infty}\tfrac{1}{n} \log\mathbb{P}_{z}(1\leq Z_n\leq e^{\theta n})\ . \nonumber
\end{align*}
\end{lemma}
\begin{proof} The first identity is given by  Proposition  \ref{varrho} (ii) and we focus on the second one. We observe that $\mathbb{P}_{z}(1\leq Z_n\leq e^{\theta n})$
decreases as $\theta$ decreases.
As for every $\theta>0$, $\mathbb{P}_z(Z_{n}=z) \leq \mathbb{P}_z(1\leq Z_n \leq e^{\theta n})$ for $n$ large enough, we have 
\begin{align}
\varrho= \lim_{n\rightarrow\infty} \tfrac{1}{n} \log \mathbb{P}_{z}(Z_{n}=z) 
\leq \lim_{\theta\rightarrow0} \liminf_{n\rightarrow\infty} \tfrac{1}{n} \log \mathbb{P}_{z}(1\leq Z_n \leq e^{\theta n}). \nonumber
\end{align} 
Let us prove the converse inequality. First, we observe 
 that $m_q<1-\varepsilon$ implies $q(0)>\varepsilon$. Using that $\P(m_Q<1)>0$ by assumption and $z\in\mathcal{I}$,  we choose $\varepsilon>0$ and $j_1\geq 1$ such
 that the sets 
 $$
 \mathcal{A}:= \{ q\in\Delta  :  q(0)>\varepsilon,\, q(z)>\varepsilon\}, \quad 
\mathcal{B}:= \{ q\in\Delta   :  m_q<1-\varepsilon,\, q(j_1)>\varepsilon\}$$
satisfy
$$\mathbb{P}(Q_1\in\mathcal{A})>0, \qquad \mathbb{P}(Q_1\in\mathcal{B})>0, \qquad \mathcal{B} \subset \{ q\in\Delta  :  q(0)>\varepsilon,\, q(z)>\varepsilon\}.$$ 
By Markov property, for every $\theta>0$,
\begin{align}
 \mathbb{P}_{z}(&Z_{n+\lfloor\frac{\theta}{\varepsilon} n\rfloor}= z)\geq \sum_{k=1}^{\lfloor e^{\theta n}\rfloor} \mathbb{P}_{z}(Z_n=k) 
\mathbb{P}_k( Z_{\lfloor \frac{\theta}{\varepsilon} n\rfloor}= z) \nonumber \\
&\geq \mathbb{P}_{z}(1\leq Z_n \leq e^{\theta n}) \min_{1\leq k \leq e^{\theta n}}\mathbb{P}_k(Z_{\lfloor \frac{\theta}{\varepsilon}n\rfloor}= z) \nonumber \\
&\geq \mathbb{P}_{z}(1\leq Z_n \leq e^{\theta n}) \min_{1\leq k \leq e^{\theta n}} 
\mathbb{E}\big[\mathbb{P}_k(Z_{\lfloor\frac{\theta}{\varepsilon}n\rfloor} = z|\mathcal{E})
; Q_1, \ldots,Q_{\lfloor\frac{\theta}{\varepsilon} n\rfloor-1} \in \mathcal{B}, Q_{\lfloor\frac{\theta}{\varepsilon} n\rfloor}\in\mathcal{A}\big] \nonumber \\
&\geq \mathbb{P}_{z}(1\leq Z_n \leq e^{\theta n}) \times \nonumber \\
& \qquad \min_{1\leq k \leq e^{\theta n}}\mathbb{E}\big[\mathbb{P}_{k-1}(Z_{\lfloor\frac{\theta}{\varepsilon} n\rfloor-1} =0|\mathcal{E})
\mathbb{P}_1(Z_{\lfloor\frac{\theta}{\varepsilon} n\rfloor} =z|\mathcal{E});
  Q_1, \ldots,Q_{\lfloor\frac{\theta}{\varepsilon} n\rfloor} \in \mathcal{B}, Q_{\lfloor\frac{\theta}{\varepsilon} n\rfloor}\in\mathcal{A}\big]\ .\label{eq25}
\end{align}
Using again the Markov property and the definition of $\mathcal{B}$ and $\mathcal{A}$, we  estimate a.s.
\begin{align}
 \mathbb{P}_{1}(Z_{\lfloor\frac{\theta}{\varepsilon} n\rfloor}= z| &Q_1, \ldots,Q_{\lfloor\frac{\theta}{\varepsilon} n\rfloor-1} \in
 \mathcal{B}, Q_{\lfloor\frac{\theta}{\varepsilon} n\rfloor}\in\mathcal{A})\nonumber \\
&\geq \mathbb{P}_{1}(Z_1 =j_{1}|Q_1\in\mathcal{B}) \cdot \mathbb{P}_{j_{1}}(Z_1 =j_{1}|Q_{1}\in\mathcal{B})^{\lfloor\frac{\theta}{\varepsilon} n\rfloor-2}\cdot 
 \mathbb{P}_{j_{1}}(Z_1=z|Q_1\in\mathcal{A}) \nonumber \\
&\geq \varepsilon\cdot\varepsilon^{j_{1}(\lfloor\frac{\theta}{\varepsilon} n\rfloor-2)}\cdot\varepsilon^{j_{1}}= \varepsilon^{j_{1}(\lfloor\frac{\theta}{\varepsilon} 
n\rfloor-1)+1}\ .\nonumber
\end{align}
Using the classical estimates $\mathbb{P}_1(Z_n>0|\mathcal{E}) \leq \exp(L_n)$ a.s., where 
\begin{align}
L_n:=\min_{0\leq k\leq n} S_k, \label{defmin}
\end{align}
 and $\log(1-x)\leq -x$, $x\in[0,1)$ yields  for every $k,n\in\mathbb{N}$
\begin{align}
 \mathbb{P}_{k}(Z_{\lfloor\frac{\theta}{\varepsilon} n\rfloor}=0|Q_1\in \mathcal{B},\ldots, Q_{n}\in \mathcal{B}) 
\geq  \big(1-e^{\lfloor\frac{\theta}{\varepsilon} n\rfloor\log(1-\varepsilon)}\big)^{k} 
\geq \big(1-e^{-\lfloor\frac{\theta}{\varepsilon} n\rfloor \varepsilon }\big)^{k}\ \text{a.s.} \nonumber
\end{align}
Inserting the two last inequalities into (\ref{eq25}), we get that
\begin{align}
 \mathbb{P}_{z}&(Z_{n+\lfloor\frac{\theta}{\varepsilon} n\rfloor}= z)\mathbb{P}_{z}(1\leq Z_n \leq e^{\theta n})^{-1} \nonumber \\
& \geq  \min_{1\leq k \leq e^{\theta n}}\Big\{\big(1-e^{-\varepsilon\lfloor\frac{\theta}{\varepsilon}n-1\rfloor}\big)^k\varepsilon^{j_{1}(\lfloor\frac{\theta}{\varepsilon} n\rfloor-1)+1}
\mathbb{P}(Q_1\in \mathcal{B}, \ldots,Q_{\lfloor\frac{\theta}{\varepsilon} n\rfloor-1} \in \mathcal{B},
Q_{\lfloor\frac{\theta}{\varepsilon} n\rfloor}\in\mathcal{A} )\Big\}\nonumber \\
&\geq (1-e^{-\theta n+o(1)})^{e^{\theta n}}\varepsilon^{j_{1}(\lfloor\frac{\theta}{\varepsilon} n\rfloor-1)+1}
\mathbb{P}(Q\in\mathcal{B})^{\lfloor\frac{\theta}{\varepsilon} n\rfloor-1}\mathbb{P}(Q\in\mathcal{A}).\nonumber 
\end{align}
Taking the logarithm and using the fact that $(1-1/x)^x$ is increasing for $x\geq 1$ and bounded
\begin{align}
(1+\theta/\epsilon)\varrho=\lim_{n\rightarrow\infty} \tfrac{1}{n} \log\mathbb{P}_z&( Z_{n+\lfloor\frac{\theta}{\varepsilon} n\rfloor}= z)\nonumber \\
&\geq \limsup_{n\rightarrow\infty}\tfrac{1}{n} 
 \log\mathbb{P}_z( 1\leq Z_n \leq e^{\theta n})+\tfrac{j_{1}\theta}{\varepsilon} \log \varepsilon+\tfrac{\theta}{\varepsilon}\log\mathbb{P}(Q\in\mathcal{B})\ .\label{boundprob1}
\end{align}
Thus, letting $\theta \rightarrow 0$,
\begin{align}
\varrho
&\geq \lim_{\theta\rightarrow 0}\limsup_{n\rightarrow\infty}\tfrac{1}{n} \log\mathbb{P}_z(1\leq Z_n \leq e^{\theta n}), \nonumber
\end{align}
which gives the expected converse inequality. 
\end{proof}

\begin{lemma}\label{prop1}
Under Assumption \ref{finvar}, for every $b>0$, $n\in\mathbb{N}$ and $r\in(0,1)$, it holds that
\begin{align}
 \mathbb{P}_b(Z_n\leq r \ e^{S_n}|\mathcal{E}) \leq \big(1-(1-r)^2 e^{L_n}/(n+2)\big)^b \qquad \text{a.s.} \nonumber
\end{align}
\end{lemma}
\begin{proof}
 Note that $\mathbb{E}[Z_n(Z_n-1)|\mathcal{E}]=f''_{0,n}(1)$ a.s. Let us now check briefly  that the result of   Proposition 1 in  \cite{BK09}  still holds, which means that we can replace 
 Assumption 2 in \cite{BK09} by our Assumption \ref{finvar}. 
From $f_{0,n} = f_{0,n-1} \circ f_n$, by chain rule for differentiation $f'_{0,n}(1)=f'_{0,n-1}(1)f'_n(1)$ and 
$f_{0,n}''(1) = f_{0,n-1}''(1) f_n'(1)^2+f_{0,n-1}'(1) f_n''(1)$, we get that
\begin{eqnarray}
 \frac{f''_{0,n}(1)}{f'_{0,n}(1)^2} &=& \frac{f''_{0,n-1}(1)}{f'_{0,n-1}(1)^2} + \frac{f_n''(1)}{f'_{0,n-1}(1) f'_n(1)^2}. \nonumber
\end{eqnarray}
Using Assumption \ref{finvar} yields
\begin{eqnarray}
\frac{f_n''(1)}{f'_{0,n-1}(1)f'_n(1)^2} &\leq& d (e^{-S_{n-1}}+e^{-S_n}). \nonumber
\end{eqnarray} 
By iterating this inequality, we have a.s.
\begin{eqnarray}
 \frac{\mathbf{E}[Z_n(Z_n-1)|\Pi]}{\mathbf{E}[Z_n|\Pi]^2} &=& \frac{f''_{0,n}(1)}{f'_{0,n}(1)^2} \ \leq \ 2d \sum_{k=0}^n e^{-S_k} \quad \text{a.s.}  \nonumber
\end{eqnarray}
Finally we get for every $n\in\mathbb{N}$, 
\begin{align}
 \mathbb{E}_1[Z_n(Z_n-1)|\mathcal{E}] &\leq 2d e^{2S_n}\sum_{k=0}^n e^{-S_k}\leq 2d\ (n+1) e^{S_n} e^{S_n-L_n}\ \text{a.s.} \nonumber
\end{align}
Combining this inequality  with an inequality due to Paley and Zygmund, which ensures that 
 for any $[0,\infty)$ valued random variable $\xi$ such that $0 < \E[\xi] < \infty$ and $0 < r < 1$, we have $\mathbb P(\xi > r\mathbb E[\xi])\geq (1-r)^2\mathbb E[\xi]^2/
\mathbb E[\xi^2]$
(see Lemma 4.1 in  \cite{kallenberg}). Then a.s., 
\begin{align}
 \mathbb{P}_1(Z_n\geq r \ e^{S_n}|\mathcal{E}) & \geq (1-r)^2 \frac{\mathbb{E}_1[Z_n|\mathcal{E}]^2}{\mathbb{E}_1[Z_n^2|\mathcal{E}]}\nonumber \\
&\geq  (1-r)^2 \frac{e^{2S_n}}{(n+1) e^{S_n} e^{S_n-L_n}+e^{S_n}} = \frac{(1-r)^2}{n+2} e^{L_n}. \nonumber 
\end{align}
Given $\mathcal{E}$ and starting with $Z_0=b$, $b$-many subtrees are developing independently. Each has the above probability of being larger than $re^{S_n}$. 
Thus
\begin{align}
 \mathbb{P}_b(Z_n\leq r \ e^{S_n}|\mathcal{E}) &\leq \mathbb{P}(Z_n\leq r \ e^{S_n}|\mathcal{E})^b \nonumber \\
&\leq  \big(1-(1-r)^2 \tfrac{e^{L_n}}{n+2}\big)^b\qquad \text{a.s.}, \nonumber
\end{align}
which is the claim of the lemma.
\end{proof}

\begin{lemma} \label{upperA3}
If $\P(X<0)>0$ and  Assumption \ref{finvar} holds, then for all $z\in Cl(\mathcal{I})$, $\theta \in \big(0,\E[X]\big]$,
$$\limsup_{n \rightarrow \infty} \tfrac{1}{n} \log \P_z(1\leq Z_{n}\leq \exp(n\theta))\leq -\chi(\theta, \varrho, \Lambda).$$ 
\end{lemma}
\begin{proof}
Let $z\in Cl(\mathcal{I})$. For the proof of the upper bound, we will decompose the probability at the first moment when there are at least $n^3$-many individuals for the rest of time. For this, let
\begin{align}
 \sigma_n:=\inf\{1\leq i\leq n : Z_{j}\geq n^3,  \ j= i, \ldots, n \}, \qquad (\inf \emptyset:=n) \nonumber
\end{align}
and 
\begin{align*}
 \tau_n:= \inf\big\{0\leq i\leq n : S_{i}\leq \min\{S_0,S_1,\ldots,S_n\}\big\}\ .
\end{align*}
Let us fix $0<\theta<\mathbb{E}[X]$. Then by Markov property,
\begin{align}
 \mathbb{P}_{z}&(1\leq Z_n\leq e^{\theta}) =\sum_{i=1}^n \mathbb{P}_{z}(\sigma_{n}=i, 1\leq Z_n\leq e^{\theta n}) \nonumber \\
&\leq \sum_{i=1}^n \mathbb{P}_{z}(1\leq Z_{i-1}< n^3) \max_{k\geq n^3} \mathbb{P}_k(1\leq Z_{n-i}\leq e^{\theta n},  \quad \forall 1\leq j\leq n-i : \ Z_j\geq n^3) \nonumber \\
&= \sum_{i=1}^n \mathbb{P}_{z}(1\leq Z_{i-1}< n^3) \sum_{j=0}^{n-i} \max_{k\geq n^3} \mathbb{P}_k(1\leq Z_{n-i}\leq e^{\theta n};\tau_{n-i}=j, \quad \forall 1\leq j\leq n-i: \  Z_j\geq n^3) \nonumber \\
&\leq \sum_{i=1}^n \mathbb{P}_z(1\leq Z_{i-1}< n^3) \sum_{j=0}^{n-i} \mathbb{P}(\tau_j=j) \max_{k\geq n^3} 
\mathbb{P}_k(1\leq Z_{n-i-j}\leq e^{\theta n};L_{n-i-j}\geq 0).\label{main1}
\end{align}
Next, we treat the different probabilities separately. First, by Lemma \ref{l_rho2}
 for all $t,s\in(0,1)$ with $s+t\leq 1$, we have
\begin{align}
\limsup_{n\rightarrow\infty} \tfrac{1}{n} \log \mathbb{P}(1\leq Z_{\lfloor (1-t-s) n \rfloor -1}\leq n^3)= -(1-t-s)\rho. \nonumber
\end{align}
As to the second probability, as $\mathbb{P}(\tau_n=n)\leq \mathbb{P}(S_n\leq 0)$,
\begin{align*}
 \lim_{n\rightarrow\infty} \tfrac{1}{n} \log \mathbb{P}(\tau_{\lfloor sn\rfloor } =\lfloor sn\rfloor ) \leq -s\Lambda(0).
\end{align*}
Next, for every $\varepsilon>0$,
\begin{align}
 \min_{k\geq n^3}\ &\mathbb{P}_k(1\leq Z_{\lfloor tn\rfloor }\leq e^{\theta n};L_{\lfloor t n\rfloor}\geq 0) \nonumber \\
&\leq\min_{k\geq n^3} \mathbb{E}\big[\mathbb{P}_k(1\leq Z_{\lfloor tn\rfloor }\leq e^{\theta n}|\mathcal{E});
 S_{\lfloor tn\rfloor }\geq (\theta+\varepsilon)n, L_{\lfloor tn\rfloor }\geq 0\big]+\mathbb{P}\big(S_{\lfloor tn\rfloor }\leq  (\theta+\varepsilon)n\big).\nonumber 
\end{align}
Using Lemma \ref{prop1}, for $n$ large enough,
\begin{align}
 \max_{k\geq n^3}&\ \mathbb{E}\big[\mathbb{P}_k(1\leq Z_{\lfloor tn\rfloor }\leq e^{\theta n}|\mathcal{E});
 S_{\lfloor tn\rfloor }\geq (\theta+\varepsilon)n, L_{\lfloor tn\rfloor }\geq 0\big]\nonumber \\
&\leq \max_{k\geq n^3}\mathbb{E}\big[\mathbb{P}_k(1\leq Z_n\leq e^{-\varepsilon n} e^{S_{\lfloor tn\rfloor}}|\mathcal{E});
 S_{\lfloor tn\rfloor}\geq (\theta+\varepsilon)n, L_{\lfloor tn\rfloor}\geq 0\big]\nonumber \\
&\leq \max_{k\geq n^3} \big(1-(1-e^{-\varepsilon n})^2 \tfrac{1}{\lfloor tn\rfloor+2}\big)^k 
\mathbb{P}\big(L_{\lfloor tn\rfloor}\geq 0,  S_{\lfloor tn\rfloor}\geq (\theta+\varepsilon)n\big)\nonumber \\
&\leq \big(1-(1-\tfrac{1}{2} )^2 \tfrac{1}{\lfloor tn\rfloor+2}\big)^{n^3} .\nonumber 
\end{align}
Then, for every $t>0$,
\begin{align*}
\limsup_{n\rightarrow\infty} \tfrac{1}{n} \log& \max_{k\geq n^3} \mathbb{E}\big[\mathbb{P}_k(1\leq Z_{\lfloor tn\rfloor }\leq e^{\theta n}|\mathcal{E});
 S_{\lfloor tn\rfloor }\geq (\theta+\varepsilon)n, L_{\lfloor tn\rfloor }\geq 0\big] \\
&\qquad \qquad \leq \limsup_{n\rightarrow\infty} n^2 \log \big(1-\tfrac{1}{4} \tfrac{1}{{\lfloor tn\rfloor+2}}\big)
=-\infty. 
\end{align*}
Finally, recall that
\begin{align}
 \lim_{n\rightarrow\infty} \tfrac{1}{n} &\log \mathbb{P}\big(S_{\lfloor tn\rfloor} \leq (\theta+\varepsilon)n \big)= -t\Lambda\big((\theta+\varepsilon)/t\big)\ . \nonumber
\end{align} 
Applying all this in (\ref{main1}) and letting $\varepsilon\rightarrow 0$ yields the upper bound, i.e.
\begin{align}
 \limsup_{n\rightarrow\infty} \tfrac{1}{n} \log \mathbb{P}(1\leq Z_n\leq e^{n\theta}) &\leq -\inf_{s,t\in[0,1]; s+t\leq 1} \big\{(1-s-t)\rho + s\Lambda(0)+t\Lambda((\theta+\varepsilon)/t)\big\} \nonumber\\
&=  -\inf_{t\in[0,1]} \{(1-t)\rho +t\Lambda(\theta/t+)\}=\chi(\theta, \varrho,\lambda).\nonumber
\end{align}
In the last step, we used that Proposition 2 in \cite{BB11}  guarantees $\Lambda(0)\geq \rho$ under Assumption \ref{finvar}, together with  $\Lambda(0)\geq \Lambda(x)$ for every $x\geq 0$ and right-continuity of $\Lambda$.
\end{proof}
 \subsection{Proof of Theorems \ref{thlower0}  and \ref{thlower}}

\begin{proof}[Proof of Theorem \ref{thlower0} (i)] Let $z\geq 1$.
The second part of Proposition  \ref{varrho}  ensures that for $b$ large enough, 
$$\varrho=-\lim_{n \rightarrow \infty} \tfrac{1}{n} \log\P_z(1\leq Z_{n}\leq b).$$ 
Then,  under Assumption \ref{as_strongly_supercrit},  Lemmas \ref{proplow} and \ref{propupp}  yield  
$$\lim_{n\rightarrow\infty} \tfrac{1}{n} \log \mathbb{P}(1\leq Z_n\leq e^{n\theta})=-\chi(\theta, \varrho, \Lambda).$$
The right-continuity of $\chi(\theta, \varrho, \Lambda)$ proves the last part of the result.
\end{proof}
\begin{proof}[Proof of Theorem \ref{thlower0} (ii)]
We recall from Proposition 2.1. $(i)$ that  for every $b\geq z$
$$\lim_{n \rightarrow \infty} \tfrac{1}{n} \log\P_z(1\leq Z_{n}\leq b)=\log \E[Q(1)^z].$$ 
Then,  under Assumption \ref{as_strongly_supercrit} and $\E[Z_1\log^+(Z_1)]<\infty$,  Lemmas \ref{proplow} and \ref{propupp}  yield  
$$\lim_{n\rightarrow\infty} \tfrac{1}{n} \log \mathbb{P}_z(1\leq Z_n\leq e^{n\theta})=-\chi\big(\theta, -\log \E[Q(1)^z], \Lambda\big).$$
The right-continuity of $\chi(\theta, \varrho, \Lambda)$ proves the last part of the result.
\end{proof}
\begin{proof}[Proof of Theorem \ref{thlower}]
The first part is a direct consequence of Lemma \ref{l_rho2}. 

As we assume $\P(X<0)>0$, we can use again the second part of Proposition  \ref{varrho}, which ensures  that for $b$ large enough, 
$$\varrho=-\lim_{n \rightarrow \infty} \tfrac{1}{n} \log\P_z(1\leq Z_{n}\leq b).$$ 
Then,  under Assumption \ref{finvar} and $\E[Z_1\log^+(Z_1)]<\infty$, we can combine   Lemmas \ref{proplow} and   \ref{upperA3} to get
$$\lim_{n\rightarrow\infty} \tfrac{1}{n} \log \mathbb{P}_z(1\leq Z_n\leq e^{n\theta})=-\chi(\theta, \varrho, \Lambda)$$
for every $z\geq 1$. It completes the proof.
\end{proof}

\subsection{The linear fractional case}\label{linearfrac}
In this section, we restrict ourselves to the case of offspring distributions with generating function of linear fractional form, i.e. 
\begin{eqnarray}
 f(s)\ =\ 1- \frac{1-s}{m^{-1}+ b\ m^{-2} (1-s)/2}\ ,  \nonumber
\end{eqnarray}
where $m=f'(1)$ and $b=f''(1)$. 
\medskip \\
\begin{proof}[Proof of Corollary \ref{cor1}]
Recall that $s\rightarrow\mathbb{E}[e^{sX}]$ is the moment generating function of $X$, which is a convex function. The result of the corollary is trivial if $\rho=\Lambda(0)$. Thus, using
(\ref{rateLF}), we can focus on the case $\rho=\mathbb{E}[e^{-X}]$ and 
 $0<\mathbb{E}[Xe^{-X}]<\infty$. Then  $\mathbb{E}[e^{-X}]<\infty$ and we have 
$\varrho=-\log\mathbb{E}[e^{-X}]\leq \sup_{s<0} \{-\log\mathbb{E}[e^{sX}]\}=\Lambda(0)$. Note that $\Lambda(0)=\infty$ is possible. \\
Let us recall some details of Legendre transforms. It is well-known (see e.g. \cite{hollander}) that
\begin{align*}
 v_\theta(s):=-\theta s - \log\mathbb{E}\big[e^{-sX}\big] 
\end{align*}
is a convex function. The conditions $\mathbb{E}[e^{-X}]<\infty$ and $0<\mathbb{E}[Xe^{-X}]<\infty$ imply by the dominated convergence theorem that 
$v$ above is differentiable in $s=1$ and 
\begin{align*}
 v'_\theta(1):=-\theta - \mathbb{E}\big[Xe^{-X}\big]/\mathbb{E}\big[e^{-X}\big] \ .
\end{align*}
Thus by definition of $\theta^*$, the derivative of $v_{\theta^*}'$ vanishes for $s=1$, i.e. $v_\theta^*$ takes its minimum in $s=1$. Thus,
\begin{align*}
 \Lambda(\theta^*):=-\theta^* - \log\mathbb{E}\big[e^{-X}\big] <\infty
\end{align*}
and by the theory of Legendre transforms, the tangent $t$ on the graph of $\Lambda$ in $\theta^*$ is described by
\begin{align*}
 t(\theta)&:=-\theta-\log\mathbb{E}\big[e^{-X}\big] \ .
\end{align*}
As $\Lambda$ is convex and decreasing for $\theta<\mathbb{E}[X]$, we have $\Lambda(\theta)\geq t(\theta)$ for $\theta<\theta^*$. This proves the representation in Corollary 
\ref{cor1}.
\end{proof}

\textbf{Acknowledgement.}   The author is grateful to Eric Miqueu for pointing out a mistake in the  previous version of this work in the expression of the speed of decrease $\varrho_k$ in  the case without extinction
$\P_1(Z_1=0)=0$.\\
This work partially was funded by project MANEGE `Mod\`eles
Al\'eatoires en \'Ecologie, G\'en\'etique et \'Evolution'
09-BLAN-0215 of ANR (French national research agency),  Chair Modelisation Mathematique et Biodiversite VEOLIA-Ecole Polytechnique-MNHN-F.X. and the professorial chair Jean Marjoulet.

\bibliographystyle{apalike}
\end{document}